\documentclass[11pt]{amsart}
\usepackage[utf8]{inputenc}
\usepackage{amssymb}
\usepackage{stmaryrd}
\usepackage{amsfonts}
\usepackage{amsmath}
\usepackage{xcolor}
\usepackage[english]{babel}
\usepackage{amsthm}
\usepackage{setspace}
\usepackage{tikz-cd}
\PassOptionsToPackage{hyphens}{url}\usepackage{hyperref}
\usepackage{graphicx}
\usepackage[perpage]{footmisc}
\usepackage[shortlabels]{enumitem}
\usepackage{longtable}
\usepackage{comment}
\usepackage{cases}
\usepackage{diagbox}
\usepackage{float}

\usepackage{mathtools}

\newtheoremstyle{example}
{}                
{}                
{\sffamily}        
{}                
{\bfseries}       
{.}               
{ }               
{}                

\usepackage[
backend=biber,
style=alphabetic
]{biblatex}
\newtheorem{theorem}{Theorem}

\newtheorem{lemma}[theorem]{Lemma}
\newtheorem{corollary}[theorem]{Corollary}

\theoremstyle{example}
\newtheorem{example}[theorem]{Example}

\theoremstyle{definition}
\newtheorem{definition}[theorem]{Definition}

\newtheorem{remark}[theorem]{Remark}

\numberwithin{theorem}{section}

\numberwithin{theorem}{section}

\DeclareMathOperator{\Gal}{Gal}

\DeclareMathOperator{\Aut}{Aut}
\DeclareMathOperator{\Disc}{Disc}

\DeclareMathOperator{\ord}{ord}

\newcommand{\co}{\mathcal{O}}

\newcommand{\Z}{\mathbb{Z}}
\newcommand{\R}{\mathbb{R}}
\newcommand{\Q}{\mathbb{Q}}
\newcommand{\C}{\mathbb{C}}
\newcommand{\F}{\mathbb{F}}
\newcommand{\N}{\mathbb{N}}
\newcommand{\p}{\mathfrak{p}}

\newcommand{\triv}{\mathrm{triv}}
\newcommand{\m}{\widetilde{m}}

\usepackage[margin=1.2in]{geometry}
\title{$S_4$-quartics with Prescribed Norms}
\author{Sebastian Monnet}
\begin{document}

\begin{abstract}
    Given a number field $k$ and a finitely generated subgroup $\mathcal{A} \subseteq k^*$, we study the distribution of $S_4$-quartic extensions of $k$ such that the elements of $\mathcal{A}$ are norms. We show that the density of such extensions is the product of so-called ``local masses'' at every place of $k$. We give these local masses explicitly in almost all cases and give an algorithm for computing the remaining cases.
\end{abstract}

\maketitle
\section{Introduction}
\label{sec-intro}
Fix a number field $k$. An \emph{$S_4$-quartic extension of $k$} is a quartic field extension $K/k$, such that the normal closure $\widetilde{K}$ of $K$ over $k$ has $\Gal(\widetilde{K}/k) \cong S_4$. For a real number $X > 0$, write $N(X)$ for the number of isomorphism classes of $S_4$-quartic extensions $K/k$ such that the norm of the relative discriminant of $K$ over $k$ is at most $X$. Given a finitely generated subgroup $\mathcal{A}$ of the unit group $k^*$, write $N(X;\mathcal{A})$ for the number of such isomorphism classes with $\mathcal{A} \subseteq N_{K/k}(K^*)$. 

\begin{theorem}
    \label{thm-prop-subgroup-between-0-1}
    For every finitely generated subgroup $\mathcal{A}\subseteq k^*$, we have 
    $$
    0 < \lim_{X\to\infty} \frac{N(X;\mathcal{A})}{N(X)} \leq 1,
    $$
    with equality if and only if $\mathcal{A}\subseteq k^{*4}$. 
\end{theorem}
\begin{corollary}
    \label{cor-exist-exts-with-subgroup}
    For every finitely generated subgroup $\mathcal{A}\subseteq k^*$, there are infinitely many $S_4$-quartic extensions $K$ of $k$ with $\mathcal{A}\subseteq N_{K/k}(K^*)$. 
\end{corollary} 
Let us unpack the content of Theorem~\ref{thm-prop-subgroup-between-0-1}. It tells us that any finitely generated subgroup $\mathcal{A}\subseteq k^*$ is contained in the norm group of a positive proportion of $S_4$-quartic extensions. Moreover, it tells us that $\mathcal{A}$ can only be contained in 100\% of such norm groups if $\mathcal{A}\subseteq k^{*4}$, in which case $\mathcal{A}\subseteq N_{K/k}(K^*)$ for every $S_4$-quartic extension $K$ of $k$. In other words, if $\mathcal{A}\subseteq N_{K/k}(K^*)$ for almost all $K$, then $\mathcal{A}\subseteq N_{K/k}(K^*)$ for all $K$.

\begin{remark}
    In the special case where $\mathcal{A}$ is generated by one element, Corollary~\ref{cor-exist-exts-with-subgroup} can be obtained by classical methods, using Hilbert's irreducibility theorem. See \cite[Example~1.13]{newton-nf-prescribed} for details.
\end{remark}
In \cite{newton-nf-prescribed}, Frei, Loughran, and Newton considered the same question for abelian extensions with a given Galois group. They used class field theory to solve the problem. In our setting, the extensions are nonabelian, so we are unable to use class field theory in the same way. Our approach is to use a local-global principle called the Hasse norm principle to reduce the global problem to a collection of local conditions. We then apply a result of Bhargava, Shankar, and Wang to count the $S_4$-quartics satisfying this collection of local conditions. Bhargava, Shankar, and Wang's result is also valid for $S_3$-cubics and $S_5$-quintics, and we expect that our methods will generalise to such extensions.

To prove Theorem~\ref{thm-prop-subgroup-between-0-1}, we first obtain the following explicit formula for the density of number fields in question.

\begin{theorem}
    \label{thm-counting-function-subgroup}
    Let $\mathcal{A}\subseteq k^*$ be a finitely generated subgroup. There exist positive rational numbers $m_{\mathcal{A},\p}$, indexed by the primes $\p$ of $k$, such that each $m_{\mathcal{A},\p}$ is defined explicitly in terms of $\mathcal{A}$ and $\p$, and 
    $$
    \lim_{X\to\infty} \frac{N(X;\mathcal{A})}{X} = \frac{1}{2}\operatornamewithlimits{Res}_{s=1}\big(\zeta_k(s)\big) \prod_\p m_{\mathcal{A},\p}. 
    $$
\end{theorem}
 
For an element $\alpha \in k^*$, write $N(X;\alpha)$ to mean $N(X;\langle \alpha \rangle)$, where $\langle\alpha\rangle$ is the subgroup of $k^*$ generated by $\alpha$. Equivalently, $N(X;\alpha)$ is the number of $S_4$-quartic extensions $K/k$ with appropriately bounded discriminant such that $\alpha \in N_{K/k}(K^*)$. Similarly, write $m_{\alpha,\p}$ for the number $m_{\langle\alpha\rangle,\p}$ from Theorem~\ref{thm-prop-subgroup-between-0-1}.
\begin{theorem}
    \label{thm-values-of-masses}
    For $\alpha \in k^*$ and finite primes $\p$ of $k$ with odd norm, the values of $m_{\alpha,\p}$ are given by Tables~\ref{table-q-1mod4} and \ref{table-q-3mod4}. If $\p$ is an infinite prime and $f:k \to \C$ is an embedding corresponding to $\p$, then 
    $$
    m_{\alpha,\p} = \begin{cases}
        \frac{5}{12} \quad \text{if $f$ is real and $f(\alpha)>0$},
        \\
        \frac{7}{24} \quad \text{if $f$ is real and $f(\alpha) < 0$},
        \\
        \frac{1}{24} \quad \text{if $f$ is complex.}
    \end{cases}
    $$
    In the special case $k=\Q$ and $\p = (2)$, the values of $m_{\alpha,\p}$ are given by Table~\ref{table-q-2}.
\end{theorem}

We illustrate how to use our results, and in particular how to read Tables~\ref{table-q-1mod4},~\ref{table-q-3mod4}, and~\ref{table-q-2}, with an example. 

\begin{example}
    We will give an exact expression for the proportion of $S_4$-quartic number fields $K/\Q$ such that $2 \in \N_{K/\Q}(K^*)$. Since $2 > 0$, we have 
    $$
    m_{2,\infty} = m_{1,\infty} = \frac{5}{12}. 
    $$
    In Table~\ref{table-q-2}, we have $r = u = 1$, so 
    $$
    m_{2,2} = \frac{6523}{8192}. 
    $$
    For an odd prime $p$, we know that $2$ is a quadratic residue if and only if $p \equiv \pm 1 \pmod{8}$. Therefore, if $p\equiv \pm 3 \pmod{8}$, we have 
    $$
    m_{2,p} = \frac{(4p^2 + 4p + 5)(p-1)}{4p^3}.
    $$
    If $p \equiv 7 \pmod{8}$, then Table~\ref{table-q-3mod4} tells us that 
    $$
    m_{2,p} = \frac{(p^3 + p^2 + 2p + 1)(p-1)}{p^4}.
    $$
    For $p \equiv 1\pmod{8}$, the situation is more complicated, and we have 
    $$
    m_{2,p} = 
    \begin{cases}
        \frac{(p^3 + p^2 + 2p + 1)(p-1)}{p^4} \quad \text{if $2$ is a quartic residue modulo $p$},
        \\
        \frac{(p^2+p +2)(p-1)}{p^3} \quad \text{otherwise}.
    \end{cases}
    $$
    We have 
    $$
    m_{1,p} = \frac{(p^3 + p^2 + 2p + 1)(p-1)}{p^4},
    $$
    for all finite primes (including $2$), and therefore 
    $$
    \frac{m_{2,p}}{m_{1,p}} = \begin{cases}
        1 \quad \text{if $p=\infty$},
        \\
        \frac{6523}{8704} \quad \text{if $p=2$},
        \\
        \frac{4p^3 + 4p^2 + 5p}{4(p^3 + p^2 + 2p + 1)} \quad \text{if $p\equiv \pm 3 \pmod{8}$},
        \\
        1 \quad \text{if $p\equiv 7\pmod{8}$},
        \\
        1 \quad \text{if $p\equiv 1\pmod{8}$ and $2$ is a quartic residue modulo $p$},
        \\
        \frac{p^3 + p^2 + 2p}{p^3 + p^2 + 2p + 1} \quad \text{if $p\equiv 1\pmod{8}$ and $2$ is a quartic nonresidue modulo $p$}
    \end{cases}
    $$
    Finally then, the proportion of $S_4$-quartic extensions $K/\Q$ with $2 \in N_{K/\Q}(K^*)$ is given by the Euler product 
    $$
    \frac{1}{2}\cdot \frac{6523}{8704}\cdot \prod_{p\equiv \pm 3 \pmod{8}} \Big(1 - \frac{3p + 4}{4p^3 + 4p^2 + 8p + 4}\Big)\cdot  \prod_{\substack{p \equiv 1 \pmod{8} \\ \text{$2$ is a quartic nonresidue modulo $p$}}} \Big(1 - \frac{1}{p^3 + p^2 + 2p + 1}\Big),
    $$
    which is approximately equal to $33.06\%$. 
\end{example}
\subsection{Structure of the paper}
In Section~\ref{sec-reducing-to-local}, we explain our approach to the problem. Our first step is to introduce the \emph{Hasse Norm Principle} in Section~\ref{subsec-HNP}, which lets us consider the situation locally at each prime of $k$. Subsequently, in Section \ref{subsec-bhargava-result}, we state Bhargava, Shankar, and Wang's counting result, which requires our local conditions to be ``acceptable''. In Section~\ref{subsec-acceptable}, we demonstrate that our conditions are indeed acceptable, allowing us to prove Theorem~\ref{thm-counting-function-subgroup}.

The goal of Section~\ref{sec-masses} is to prove Theorem~\ref{thm-values-of-masses}. That is, we compute the masses $m_{\alpha,\p}$ explicitly for $\alpha \in k^*$ and primes $\p$ that are either infinite or have odd norm, as well as in the special case $k=\Q$ and $\p=(2)$. In Section~\ref{subsec-strat}, we introduce some notation and outline the strategy of Sections~\ref{subsec-norms}--\ref{subsec-primes-over-2}. Then, in Section~\ref{subsec-norms}, we compute all possible norm groups of several classes of quartic \'etale algebras over $k_\p$, where $\p$ is a finite prime with odd norm. We apply these norm group computations in Section~\ref{subsec-masses} to compute all possible masses $m_{\alpha,\p}$, for finite primes $\p$ with odd norm, as well as for infinite primes. In Section~\ref{subsec-primes-over-2}, we present our MAGMA (\cite{magma}) algorithm to compute masses in general and obtain the values in Table~\ref{table-q-2}. We then assemble our results to prove Theorem~\ref{thm-values-of-masses}.

Finally, Section~\ref{sec-proportion} is dedicated to proving Theorem~\ref{thm-prop-subgroup-between-0-1}. Along the way, we obtain an explicit finite bound on $\lim_{X\to\infty}\frac{N(X;\alpha)}{X}$ for $\alpha \in k^*$, in terms of $\alpha$ and $k$. This bound implies that $\lim_{X\to\infty}\frac{N(X,\mathcal{A})}{X}$ is finite for any finitely generated subgroup $\mathcal{A}\subseteq k^*$, which allows us to obtain an explicit expression for the proportion $\lim_{X\to\infty}\frac{N(X;\mathcal{A})}{N(X)}$. By considering this explicit expression, we prove Theorem~\ref{thm-prop-subgroup-between-0-1}.
\subsection{Acknowledgements}
I would like to thank my supervisor, Rachel Newton, for suggesting the project and for her unwavering support throughout. Jiuya Wang's suggestion to look at \cite{geom-of-nums} was vital to the eventual approach, so I am very grateful to her as well. Thanks also to Ross Patterson for suggesting I use MAGMA to verify my results; my code did actually uncover a mistake, so the check was very useful. Ashvin Swaminathan was very generous in explaining some details of Bhargava's parametrisation of quartic rings, for which I am also grateful. Of course, none of this work would have been possible without my generous funding from the ESPRC, University College London and the Heilbronn Institute for Mathematical Research. 

\subsection{Notation}
We fix the following notation: 
\begin{itemize}
    \item $k$ is a number field,
    \item $\p$ is a prime of $k$,
    \item $\F_\p$ is the residue field $\co_K/\p$, in the case where $\p$ is finite,
    \item $N(\p)$ and $q$ both denote the cardinality of $\F_\p$, in the case where $\p$ is finite,
    \item $K$ is a finite-degree extension of $k$,
    \item $\alpha$ is an element of $k^*$,
    \item $\mathcal{A}$ is a finitely generated subgroup of $K^*$,
    \item $\zeta_m$ is a primitive $m^\mathrm{th}$ root of unity, where $m$ is a positive integer. 
\end{itemize}

\subsection{Detailed version of the paper}
This version of the paper has been optimised for concision. In case the reader is interested in more detailed proofs and worked examples, there is a much longer draft available on the author's website (\url{www.smonnet.com}).

\section{Reducing to Local Conditions}
\label{sec-reducing-to-local}
In this section, we break our global problem into infinitely many local problems. In Section~\ref{subsec-HNP}, we recall the Hasse Norm Principle, which says that $\alpha$ is a norm globally if and only if it is a norm everywhere locally. In Section~\ref{subsec-bhargava-result}, we recall a result of Bhargava, Shankar, and Wang, which allows us to count number fields satisfying a given collection of local conditions. In a probabilistic sense, this counting result essentially says that the local conditions at each prime behave like independent random events. The result also tells us how to compute the ``probabilities'' of these independent events. There is one technical hitch, which is that this result only applies to so-called \emph{acceptable} collections of local conditions. Therefore, we spend Section~\ref{subsec-acceptable} proving that our local conditions are indeed acceptable, so that we can apply Bhargava, Shankar, and Wang's result.  
\subsection{The Hasse Norm Principle}
\label{subsec-HNP}
As usual, let $K/k$ be an extension of number fields and let $\p$ be a prime of $k$. Write $K_\p$ for the \'etale $k_\p$-algebra
$$
K_\p = K \otimes_k k_\p,
$$
which we call the \emph{completion of $K$ at $\p$}. In this paper, we will refer to the following result as ``the Hasse Norm Principle''. This terminology is slightly nonstandard; in general, the Hasse Norm Principle is a local-global principle that may or may not hold in a particular extension. Theorem~\ref{thm-HNP} says that this principle holds for all $S_4$-quartic extensions. 
\begin{theorem}[Hasse Norm Principle]
    \label{thm-HNP}
    Let $K/k$ be an degree $n$ extension whose normal closure has Galois group $S_n$, and let $\alpha \in k^*$. The following are equivalent: 
    \begin{enumerate}
        \item $\alpha \in N_{K_\p/k_\p}(K_\p^*)$ for each prime $\p$ of $k$,
        \item $\alpha \in N_{K/k}(K^*)$. 
    \end{enumerate} 
\end{theorem}
\begin{proof}
    This is \cite[Corollary~to~Theorem~4]{Voskresenski}. 
\end{proof}
\begin{corollary}
    \label{cor-HNP}
    Let $K/k$ be a degree $n$ extension whose normal closure has Galois group $S_n$, and let $\mathcal{A}\subseteq k^*$ be a finitely generated subgroup. The following are equivalent:
    \begin{enumerate}
        \item $\mathcal{A}\subseteq N_{K_\p/k_\p}(K_\p^*)$ for each prime $\p$ of $k$,
        \item $\mathcal{A}\subseteq N_{K/k}(K^*)$.
    \end{enumerate}
\end{corollary}

Let $\p$ be a prime of $k$. Define $\Sigma_{\mathcal{A},\p}$ to be the set of isomorphism classes of quartic \'etale algebras $L/k_\p$ with $\mathcal{A} \subseteq N_{L/k_\p}(L^*)$. Then Corollary~\ref{cor-HNP} tells us that an $S_4$-quartic extension $K/k$ has $\mathcal{A} \subseteq N_{K/k}(K^*)$ if and only if $K_\p \in \Sigma_{\mathcal{A},\p}$ for all $\p$. Given $\alpha\in k^*$, write $\Sigma_{\alpha,\p}$ for $\Sigma_{\langle\alpha\rangle, \p}$, where $\langle\alpha\rangle$ is the subgroup of $k^*$ generated by $\alpha$. 
\subsection{Counting number fields using local conditions}
The Hasse Norm Principle tells us that the number fields $K$ we are looking for are precisely those satisfying a certain collection of local conditions, in the sense of the following definition. 
\label{subsec-bhargava-result}
\begin{definition}
    Let $n$ be a positive integer. A \emph{degree $n$ collection of local conditions on $k$} is a collection $(\Sigma_\p)_\p$, where $\p$ ranges over primes of $k$, and each $\Sigma_\p$ is a set of isomorphism classes of \'etale $k_\p$-algebras of degree $n$.
\end{definition}
\begin{definition}
    An extension $K/k$ \emph{satisfies} a collection of local conditions $(\Sigma_\p)_\p$ if $K_\p \in \Sigma_\p$ for all $\p$. 
\end{definition}
\begin{remark}
    \label{remark-disc-def}
    We will refer extensively to the discriminant $\Disc(L/k_\p)$, where $L$ is an \'etale $k_\p$-algebra. This discriminant has the well-known property that
    $$
    \Disc\big((L_1\times\ldots\times L_r) / k_\p\big) = \prod_{i=1}^r \Disc(L_i/k_\p),
    $$
    for \'etale $k_\p$-algebras $L_i$. 
\end{remark}
\begin{definition}
    A \emph{$p$-adic field} is a finite degree extension of the $p$-adic numbers $\Q_p$. 
\end{definition}
\begin{definition}
    Let $F$ be a $p$-adic field with maximal ideal $P$ and residue field of size $q$. Let $L$ be an \'etale algebra over $F$. Write \emph{$\Disc(L/F)$} for the relative discriminant of $L$ over $F$, and define \emph{$\Disc_P(L/F)$} by
    $$
    \Disc_P(L/F) = q^{v_P(\Disc(L/F))}. 
    $$ 
\end{definition}
\begin{definition}
    Following \cite{geom-of-nums}, a degree $n$ collection of local conditions $(\Sigma_\p)_\p$ is said to be \emph{acceptable} if, for all but finitely many $\p$, the set $\Sigma_\p$ contains every degree $n$ \'etale algebra $L/k_\p$ with $v_{\p}(\Disc(L/k_\p))< 2$.  
\end{definition}
\begin{definition}
    Let $\Sigma = (\Sigma_\p)_\p$ be a degree $4$ collection of local conditions on $k$. For a positive real number $X$, write $N(X;\Sigma)$ for the number of $S_4$-quartic extensions of $k$ satisfying $\Sigma$ such that the norm of the relative discriminant of $K/k$ is at most $X$. 
\end{definition}

\begin{definition}
    \label{def-mass}
    Let $\p$ be a prime of $k$ and let $\Sigma_\p$ be a set of isomorphism classes of \'etale $k_\p$-algebras. Define the \emph{mass of $\Sigma_\p$} to be 
    $$
    m_\p(\Sigma_\p) = \begin{dcases}
        \frac{N(\p) - 1}{N(\p)}\sum_{L \in \Sigma_\p}\frac{1}{\Disc_\p(L/k_\p)} \cdot \frac{1}{\#\Aut(L/k_\p)} \quad \text{if $\p$ is finite},
        \\
        \sum_{L \in \Sigma_\p}\frac{1}{\#\Aut(L/k_\p)}\quad \quad\quad\quad \quad\quad\quad \quad\quad\quad \quad\text{ if $\p$ is infinite}.
    \end{dcases}
    $$
\end{definition}

\begin{theorem}
    \label{thm-bhargava-result}
    Let $\Sigma = (\Sigma_\p)_\p$ be an acceptable collection of local conditions. Then 
    $$
    \lim_{X\to\infty}\frac{N(X;\Sigma)}{X} = \frac{1}{2}\operatornamewithlimits{Res}_{s=1}\big(\zeta_k(s)\big)\prod_\p m_\p(\Sigma_\p). 
    $$
\end{theorem}
\begin{proof}
    This is \cite[Theorem~2]{geom-of-nums}, in the case where $n=4$ and $F$ is a number field. 
\end{proof}

\subsection{Showing that our conditions are acceptable}
In order to apply Theorem~\ref{thm-bhargava-result}, we need to show that the collection $(\Sigma_{\mathcal{A},\p})_\p$ of local conditions is acceptable. 
\label{subsec-acceptable}
We start by defining the ``splitting symbol'' of a prime in an extension, which will allow us to consider extensions $K/k$ separately depending on how the prime $\p$ splits in $K$.
\begin{definition}
    \label{def-splitting-symbol}
    Let $K/k$ be an extension of number fields and let $\p$ be a prime of $k$. Suppose that 
    $$
    \p\co_K = P_1^{e_1}\ldots P_r^{e_r},
    $$
    for distinct prime ideals $P_i$ of $\co_K$. For each $i$, let $f_i$ be the inertia degree of $P_i$ over $\p$. Then the \emph{splitting symbol of $\p$ in $K$} is 
    $$
    (K,\p) = (f_1^{e_1}\ldots f_r^{e_r}). 
    $$
    If $e_i=1$ for some $i$, we suppress the notation and write $f_i$ instead of $f_i^1$. Note that this symbol is only well-defined up to permutation of the indices, so we identify the symbols $(f_1^{e_1}\ldots f_r^{e_r})$ and $(f_{\sigma(1)}^{e_{\sigma(1)}}\ldots f_{\sigma(r)}^{e_{\sigma(r)}})$ for each permutation $\sigma$ of the set $\{1,2,\ldots, r\}$. For example, we say that $(1^22) = (21^2)$. 
\end{definition}

We are interested in the completion $K_\p$ of $K$ at $\p$, which is an \'etale $k_\p$-algebra. There is a natural notion of splitting symbol for \'etale $k_\p$-algebras. We first define this splitting symbol.
\begin{definition}
    Let $F$ be a $p$-adic field with maximal ideal $P$, and let $L/F$ be an \'etale algebra. Then $L \cong L_1\times \ldots \times L_r$ for finite field extensions $L_i/F$. Let $e_i$ and $f_i$ be respectively the ramification index and inertia degree of $L_i$ over $F$. Then the \emph{splitting symbol of $L$} is 
    $$
    (L,P) = (f_1^{e_1}\ldots f_r^{e_r}).
    $$
    Again, we identify splitting symbols related by permutation, in the sense of Definition~\ref{def-splitting-symbol}. 
\end{definition}
It is easy to see that our two definitions of splitting symbol agree, in the sense that 
$$
(K_\p,\p) = (K,\p),
$$
whenever $\p$ is a finite prime of $k$. The following result is standard algebraic number theory. 

\begin{lemma}
    \label{lem-disc-from-symbol}
    Let $E/F$ be an extension of $p$-adic fields, and let $P$ be the maximal ideal of $\co_F$. Let $p$ be the residue characteristic of $F$. Write $e$ and $f$ for the ramification index and inertia degree, respectively, of $E/F$. If $p\nmid e$, then 
    $$
    v_P(\Disc(E/F)) = f(e-1). 
    $$
\end{lemma}

\begin{definition}
    Let $\sigma$ be a splitting symbol. We say that $\sigma$ is \emph{overramified} if it is one of $(1^21^2), (2^2)$, and $(1^4)$. 
\end{definition}

\begin{lemma}
    \label{lem-small-disc-implies-not-overram}
    Let $F$ be a $p$-adic field with maximal ideal $P$, such that the residue characteristic is not $2$. Let $L/F$ be a quartic \'etale algebra. If $v_P(\Disc(L/F)) < 2$, then the splitting symbol $(L,P)$ is not overramified. 
\end{lemma}
\begin{proof}
    It suffices to show that overramified extensions have $v_P(\Disc(L/F)) \geq 2$. This follows from Lemma~\ref{lem-disc-from-symbol} and Remark~\ref{remark-disc-def}, since the residue characteristic is odd, hence does not divide any of the ramification indices in overramified symbols. 
\end{proof}
We will use the following well-known fact. 
\begin{lemma}
    \label{lem-unram-norm-grp-contains-units}
    Let $E/F$ be an unramified extension of $p$-adic fields. Then 
    $$
    \co_F^* \subseteq N_{E/F}(E^*).
    $$ 
\end{lemma}
\begin{proof}
    This is precisely the corollary on Page 50 of \cite{lang}. 
\end{proof}
\begin{lemma}
    \label{lem-triv-symbols-have-full-norm-group}
    Let $L/k_\p$ be a quartic \'etale algebra such that  
    $$
    (L,\p) \in \{(1111), (112), (13), (1^211), (1^22), (1^31)\}.
    $$
    We have 
    $$
    N_{L/k_\p}(L^*) = k_\p^*. 
    $$
\end{lemma}
\begin{proof}
    Write $L = L_1\times \ldots \times L_r$ as usual. If $\sigma \neq (1^22)$, then one of the $L_i$ is equal to $k_\p$, so $N_{L/k_\p}(L^*) = k_\p^*$. Suppose that $\sigma = (1^22)$. Then we may assume that $L_1/k_\p$ is quadratic and unramified, and $L_2/k_\p$ is quadratic and totally ramified. Since $L_1/k_\p$ is unramified, Lemma~\ref{lem-unram-norm-grp-contains-units} tells us that the group $N_{L_1/k_\p}(L_1^*)$ contains all units of $\co_{k_\p}$. Since $L_2/k_\p$ is totally ramified, $N_{L_2/k_\p}(L_2^*)$ contains an element of valuation $1$. The result follows. 
\end{proof}
\begin{definition}
In light of Lemma~\ref{lem-triv-symbols-have-full-norm-group}, write \emph{$\triv$} for the set
$$
\triv = \{(1111), (112), (13), (1^211), (1^31), (1^22)\}.
$$
We call the elements of this set the \emph{trivial} splitting symbols. 
\end{definition}
\begin{remark}
    We have chosen to call these symbols trivial because the problem we want to solve, namely determining whether $\mathcal{A} \subseteq N_{L/k_\p}(L^*)$, is trivial whenever $(L,\p) \in \triv$, since the answer is always yes. 
\end{remark}

\begin{lemma}
    \label{lem-acceptable-precursor}
    Let $\alpha \in k^*$, and let $\p$ be a finite prime of $k$ such that $N(\p)$ is odd and $4 \mid v_\p(\alpha)$. Then $\Sigma_{\alpha,\p}$ contains every quartic \'etale algebra $L/k_\p$ with $v_\p(\Disc(L/k_\p)) < 2$. 
\end{lemma}
\begin{proof}
    Let $L/k_\p$ be a quartic \'etale algebra with $v_\p(\Disc(L/k_\p))<2$. By Lemma~\ref{lem-small-disc-implies-not-overram}, the symbol $(L,\p)$ is not overramified, which means that it is either trivial or one of $(22), (4)$. Suppose that $(L,\p)$ is trivial. Then Lemma~\ref{lem-triv-symbols-have-full-norm-group} tells us that that $\alpha \in N_{L/k_\p}(L^*)$. Suppose instead that $(L,\p)$ is one of $(22)$ and $(4)$. Then $N_{L/k_\p}(L^*)$ contains all units by Lemma~\ref{lem-unram-norm-grp-contains-units}. Since $4\mid v_\p(\alpha)$, it follows that $\alpha \in N_{L/k_\p}(L^*)$. In either case, we have shown that $\alpha \in N_{L/k_\p}(L^*)$, so indeed $L \in \Sigma_{\alpha,\p}$. 
\end{proof}
The next result is immediate from the definitions of $\Sigma_{\mathcal{A},\p}$ and $\Sigma_{\alpha,\p}$. 
\begin{lemma}
    \label{lem-sigma-of-subgroup-eq-intersection}
    Let $\mathcal{A}\subseteq k^*$ be a finitely generated subgroup, and let $\alpha_1,\ldots, \alpha_n$ be a finite set of generators for $\mathcal{A}$. Then 
    $$
    \Sigma_{\mathcal{A},\p} = \bigcap_{i=1}^n \Sigma_{\alpha_i,\p}. 
    $$
\end{lemma}
\begin{lemma}
    \label{lem-acceptable}
    For each finitely generated subgroup $\mathcal{A} \subseteq k^*$, the collection of local conditions $(\Sigma_{\mathcal{A},\p})_\p$ is acceptable.
\end{lemma}
\begin{proof}
    Take a finite generating set $\alpha_1,\ldots, \alpha_n$ for $\mathcal{A}$, and let $S$ be the set of primes $\p$ of $k$ such that at least one of the following holds: 
    \begin{enumerate}
        \item $N(\p)$ is even,
        \item $\p$ is infinite,
        \item $4\nmid v_\p(\alpha_i)$ for some $i$. 
    \end{enumerate}  
    By Lemma~\ref{lem-acceptable-precursor}, for every prime $\p$ not in $S$, each set $\Sigma_{\alpha_i,\p}$ contains every quartic \'etale algebra $L/k_\p$ with $v_\p(\Disc(L/k_\p)) < 2$. By Lemma~\ref{lem-sigma-of-subgroup-eq-intersection}, it follows that, for every prime $\p$ not in $S$, the set $\Sigma_{\mathcal{A},\p}$ contains every quartic \'etale algebra $L/k_\p$ with $v_\p(\Disc(L/k_\p)) < 2$.
    
    Since $S$ is finite, this means that the collection $(\Sigma_{\mathcal{A},\p})_\p$ of local conditions is acceptable. 
\end{proof}
Write \emph{$m_{\mathcal{A},\p}$} for the mass $m_\p(\Sigma_{\mathcal{A},\p})$. 
\begin{proof}[Proof of Theorem~\ref{thm-counting-function-subgroup}]
    The equality is immediate from Theorem~\ref{thm-bhargava-result} and Lemma~\ref{lem-acceptable}. The numbers $m_{\mathcal{A},\p}$ are positive since $\Sigma_{\mathcal{A},\p}$ contains every extension with trivial splitting symbol, and they are rational by definition of masses in Definition~\ref{def-mass}. 
\end{proof}

\section{Computation of Masses}
\label{sec-masses}
In this section, given an $\alpha \in k^*$, we give explicit values for the masses $m_{\alpha,\p}$, whenever $N(\p)$ is odd. We also present an algorithm for computing these masses in general, which we use to calculate every possible $m_{\alpha,\p}$ in the case $k=\Q$ and $\p=(2)$. In Section~\ref{subsec-strat}, we define some notation and outline our strategy for computing the masses $m_{\alpha,\p}$. In Section~\ref{subsec-norms}, we compute explicitly the norm groups of all possible quartic \'etale algebras over $k_\p$, for $\p$ not lying over $2$. In Section~\ref{subsec-masses}, we compute all possible sets $\Sigma_{\alpha,\p}$, again for $\p$ not lying over $2$, as well as their masses $m_\p(\Sigma_{\alpha,\p})$. Finally, in Section~\ref{subsec-primes-over-2}, we describe our algorithm for computing $m_\p(\Sigma_{\alpha,\p})$ when $\p$ lies over $2$. We conclude the section by assembling the pieces to prove Theorem~\ref{thm-values-of-masses}.
\subsection{Strategy}
\label{subsec-strat}
We start by introducing some auxiliary notation. 
\begin{definition} 
    \label{defi-pre-mass}
    Let $\p$ be a finite prime of $k$, and let $\Sigma_\p$ be a set of quartic \'etale $k_\p$-algebras. Define the \emph{pre-mass of $\Sigma_\p$}, denoted \emph{$\m_\p(\Sigma_\p)$}, by 
$$
\m_\p(\Sigma_\p) = \sum_{L \in \Sigma_\p}\frac{1}{\Disc_\p(L/k_\p)}\cdot \frac{1}{\#\Aut(L/k_\p)}. 
$$
\end{definition} 
The additional notation of pre-mass is perhaps a little cumbersome, but it is a natural quantity to consider and makes the results of the current section easier to state. Of course, the point is that 
$$
m_\p(\Sigma_\p) = \frac{N(\p)-1}{N(\p)}\cdot \m_\p(\Sigma_\p)
$$
for $\p$ and $\Sigma_\p$ as in Definition~\ref{defi-pre-mass}. As such, for finite primes $\p$, we will compute $\m_\p(\Sigma_{\alpha,\p})$, and this information will immediately allow us to compute the values of $m_\p(\Sigma_{\alpha,\p})$.  
\begin{definition}
    Recall that, for $\alpha \in k^*$, we defined $\Sigma_{\alpha,\p}$ to be the set of isomorphism classes of quartic \'etale $k_\p$-algebras $L$ with $\alpha \in N_{L/k_\p}(L^*)$. For a possible splitting symbol $\sigma$, write \emph{$\Sigma_{\alpha,\p}^\sigma$} for the set 
    $$
    \Sigma_{\alpha,\p}^\sigma = \{L \in \Sigma_{\alpha,\p} : (L,\p) = \sigma\}.
    $$
\end{definition}
Since the sets $\Sigma_{\alpha,\p}^\sigma$ partition $\Sigma_{\alpha,\p}$ we have
$$
\m_\p(\Sigma_{\alpha,\p}) = \sum_\sigma \m_\p(\Sigma_{\alpha,\p}^\sigma).
$$
Therefore, it suffices to compute the values of $\m_\p(\Sigma_{\alpha,\p}^\sigma)$ for each $\sigma$. This is more tractable, since there are not many quartic \'etale algebras with a given splitting symbol, so we can consider each one in turn. In the current section, for each $\p$ with $N(\p)$ odd, we explicitly compute the sets $\Sigma_{\alpha,\p}^\sigma$ along with the norm group of each \'etale algebra in $\Sigma_{\alpha,\p}^\sigma$.

\subsection{Norm group computations}
\label{subsec-norms}
Recall that we fixed a number field $k$ and a prime ideal $\p$ of $k$, and that we write $N(\p)$ or $q$ for the size of the residue field $\co_k/\p$.

For each $\sigma$, we compute the norm group of each \'etale $k_\p$-algebra with splitting symbol $\sigma$. Throughout this subsection, $L$ will denote a quartic \'etale $k_\p$-algebra, and $\pi$ will be a uniformiser of $k_\p$. The following lemma is well-known class field theory. 
\begin{lemma}
    \label{lem-index-of-norm-group}
    Let $E/F$ be an extension of $p$-adic fields, and write $E^\mathrm{ab}$ for the largest abelian extension of $F$ contained in $E$. Then 
    $$
    [F^* : N_{E/F}(E^*)]  = [E^\mathrm{ab} : F]. 
    $$
\end{lemma}

\begin{lemma}
    \label{lem-tot-tamely-ram-exts}
    Let $F$ be a $p$-adic field with residue field of order $q$ and uniformiser $\pi$, and let $e$ be a positive integer coprime\footnote{So that our extensions are tamely ramified.} to $q$. Let $g = \gcd(e, q-1)$. There are $g$ isomorphism classes of totally ramified degree $e$ extensions of $F$, and they are given by 
    $$
    \frac{F[X]}{(X^e - \zeta_{q-1}^j\pi)},\quad j=0,1,\ldots, g-1, 
    $$
    where $\zeta_{q-1}$ is a primitive $(q-1)^\mathrm{st}$ root of unity in $F$. 
\end{lemma}
\begin{proof}
    This is almost the same as \cite[Theorem~7.2]{compute-extensions}. The only difference is that we replace the polynomial $X^e + \zeta_{q-1}^j\pi$ with $X^e - \zeta_{q-1}^j\pi$. This modification does not affect the proof. 
\end{proof}
\begin{definition} 
Retaining the notation of Lemma~\ref{lem-tot-tamely-ram-exts}, we will write \emph{$F\Big(\sqrt[e]{\zeta_{q-1}^j\pi}\Big)$} for the extension 
$$
\frac{F[X]}{(X^e - \zeta_{q-1}^j\pi)},
$$
where formally
$$
\sqrt[e]{\zeta_{q-1}^j\pi} = X + (X^e - \zeta_{q-1}^j\pi).
$$
\end{definition}

\begin{lemma}
    \label{lem-norm-group-(22)}
    If $(L,\p) = (22)$, then 
    $$
    N_{L/k_\p}(L^*) = \{v\pi^{2m} : v \in \co_{k_\p}^*, \quad m \in \Z\}.
    $$
\end{lemma}
\begin{proof}
    We have $L = L_1 \times L_1$, where $L_1$ is the unique unramified quadratic extension of $k_\p$. By Lemma~\ref{lem-unram-norm-grp-contains-units}, $N_{L_1/k_\p}(L_1^*)$ contains $\co_{k_\p}^*$. We have $N_{L_1/k_\p}(\pi) = \pi^2$, so 
    $$
    N_{L/k_\p}(L^*) = N_{L_1/k_\p}(L_1^*)\supseteq \{v\pi^{2m} : v \in \co_{k_\p}^*, \quad m \in \Z\},
    $$
    and we have equality by Lemma~\ref{lem-index-of-norm-group}. 
\end{proof}
\begin{lemma}
    \label{lem-norm-group-(4)}
    If $(L,\p) = (4)$, then 
    $$
    N_{L/k_\p}(L^*) = \{v\pi^{4m} : v \in \co_{k_\p}^*, \quad m \in \Z\}.
    $$
\end{lemma}
\begin{proof}
    The proof is essentially the same as that of Lemma~\ref{lem-norm-group-(22)}. 
\end{proof}
\begin{lemma}
    \label{lem-norm-group-(1^21^2)}
    If $N(\p)$ is odd and $(L,\p) = (1^21^2)$, then the possibilities for $L$, and its norm group, are as follows.
    \begin{center}
    \begin{tabular}{|c|c|}
        \hline
        $L$ & $N_{L/k_\p}(L^*)$
        \\
        \hline 
        \hline 
        $k_\p(\sqrt{\pi}) \times k_\p(\sqrt{\pi})$ & $\{v^2(-\pi)^m : v \in \co_{k_\p}^*, m \in \Z\}$
        \\
        \hline 
        $k_\p(\sqrt{\pi}) \times k_\p(\sqrt{\zeta_{q-1}\pi})$ & $k_\p^*$
        \\
        \hline 
        $k_\p(\sqrt{\zeta_{q-1}\pi}) \times k_\p(\sqrt{\zeta_{q-1}\pi})$ & $\{v^2(-\zeta_{q-1}\pi)^m : v \in \co_{k_\p}^*, m \in \Z\}$
        \\
        \hline 
    \end{tabular}
\end{center}
\end{lemma}
\begin{proof}
    Since $N(\p)$ is odd, Lemma~\ref{lem-tot-tamely-ram-exts} tells us that there are exactly two totally ramified quadratic extensions of $k_\p$, namely $k_\p(\sqrt{\pi})$ and $k_\p(\sqrt{\zeta_{q-1}\pi})$. Therefore, the possibilities for $L$ are indeed those listed in the table. It is easy to see that, for $j=0,1$, we have 
    $$
    N_{k_\p\big(\sqrt{\zeta_{q-1}^j\pi}\big)/k_\p}\Big(k_\p\Big(\sqrt{\zeta_{q-1}^j\pi}\Big)^*\Big) \supseteq \{v^2(-\zeta_{q-1}^j\pi)^m : v \in \co_{k_\p}^*, \quad m \in \Z\},
    $$
    and Lemma~\ref{lem-index-of-norm-group} tells us that we have equality in both cases. The result then follows from the fact that 
    $$
    N_{L/k_\p}(L^*) = N_{L_1/k_\p}(L_1^*)\cdot N_{L_2/k_\p}(L_2^*),
    $$
    where $L_1$ and $L_2$ are the quadratic extensions of $k_\p$ with $L = L_1 \times L_2$. 
\end{proof}

\begin{lemma}
    \label{lem-norm-group-(2^2)}
    If $N(\p)$ is odd and $(L,\p) = (2^2)$, then the possibilities for $L$, and the corresponding groups $N_{L/k_\p}(L^*)$, are given by the following table:
    \begin{center}
        \begin{tabular}{|c|c|}
            \hline
            $L$ & $N_{L/k_\p}(L^*)$
            \\
            \hline 
            \hline 
            $k_\p(\sqrt{\zeta_{q-1}}, \sqrt{\pi})$ & $\{v^2\pi^{2m} : v \in \co_{k_\p}^*, m \in \Z\}$
            \\
            \hline 
            $k_\p(\sqrt{\zeta_{q-1}}, \sqrt{\zeta_{q^2-1}\pi})$ & $\{v^2\zeta_{q-1}^m\pi^{2m} : v \in \co_{k_\p}^*, m \in \Z\}$
            \\
            \hline 
        \end{tabular}
    \end{center}
\end{lemma}
\begin{proof}
    By applying Lemma~\ref{lem-tot-tamely-ram-exts} to the maximal unramified subextension 
    $$
    L^{ur} = k_\p(\sqrt{\zeta_{q-1}})
    $$ 
    of $L$, we see that $L$ is one of
    $$
    L_j = L^{ur}\Big(\sqrt{\zeta_{q^2-1}^j\pi}\Big), \quad j=0,1. 
    $$
    Every element of $\co_{k_\p}^*$ is square in $L^{ur}$, and it follows that $\co_{k_\p}^{*2} \subseteq N_{L_j/k_\p}(L_j^*)$ for each $j$. We have $\pi^2\zeta_{q-1}^j \in N_{L_j/k_\p}(L_j^*)$ for $j=0,1$, which means that 
    $$
    \{v^2\zeta_{q-1}^{jm}\pi^{2m}: v \in \co_{k_\p}^*, \quad m \in \Z\} \subseteq N_{L_j/k_\p}(L_j^*),\quad j=0,1.
    $$
    Finally, Lemma~\ref{lem-index-of-norm-group} tells us that these inclusions are in fact equalities. 
\end{proof}
\begin{lemma}
    \label{lem-norm-group-(1^4)-1mod4}
    Let $(L,\p) = (1^4)$ and $q\equiv 1\pmod{4}$. Up to isomorphism, there are four distinct possibilities for $L$, given by $L_j = k_\p\Big(\sqrt[4]{\zeta_{q-1}^j\pi}\Big)$ for $j \in \{0,1,2,3\}$, and
    $$
    N_{L_j/k_\p}(L_j^*) = \{v^4(-\zeta_{q-1}^j\pi)^m : v \in \co_{k_\p}^*, m \in \Z\}.
    $$
\end{lemma}
\begin{proof}
    The classification of extensions comes from Lemma~\ref{lem-tot-tamely-ram-exts}. Clearly $N_{L_j/k_\p}(L_j^*)$ contains all fourth powers in $k_\p$, and we have $N_{L_j/k_\p}\Big(\sqrt[4]{\zeta_{q-1}^j\pi}\Big) = - \zeta_{q-1}^j\pi$, so 
    $$
    \{v^4(-\zeta_{q-1}^j\pi)^m : v \in \co_{k_\p}^*, m \in \Z\} \subseteq N_{L_j/k_\p}(L_j^*),
    $$
    and we get equality from Lemma~\ref{lem-index-of-norm-group}.
\end{proof}
\begin{lemma}
    \label{lem-norm-group-(1^4)-3mod4}
    Let $(L,\p) = (1^4)$ and $q\equiv 3\pmod{4}$. Up to isomorphism, there are two distinct possibilities for $L$, given by $L_j = k_\p\Big(\sqrt[4]{\zeta_{q-1}^j\pi}\Big)$ for $j \in \{0,1\}$, and 
    \begin{align*}
    N_{L_j/k_\p}(L_j^*) &= \{v^2(-\zeta_{q-1}^j\pi)^m : v \in \co_{k_\p}^*, m \in \Z\}
    \\
    &= \{v^4(-\zeta_{q-1}^j\pi)^m : v \in \co_{k_\p}^*, m \in \Z\}
    \end{align*}
\end{lemma}
\begin{proof}
    The proof is essentially the same as that of Lemma~\ref{lem-norm-group-(1^4)-1mod4}, except that the subgroup has index $2$, because the extensions are nonabelian. Since $q\equiv 3\pmod{4}$, we have $\co_{k_\p}^{*2} = \co_{k_\p}^{*4}$.
\end{proof}
\subsection{Mass computations}
\label{subsec-masses}
\begin{definition}
    \label{def-aut-and-disc-of-symbol}
    Let $\sigma = (f_1^{e_1}\ldots f_r^{e_r})$ be the splitting symbol of an \'etale algebra $L$ over a $p$-adic field $F$ with residue field of size $q$. Define 
    $$
    \Disc_\p(\sigma) = q^{\sum f_i(e_i-1)}
    $$
    and 
    $$
    \#\Aut(\sigma) = \Big(\prod_i f_i \Big)\cdot \#\{\varphi \in S_r : f_{\varphi(i)}^{e_{\varphi(i)}} = f_i^{e_i} \text{ for all $i$}\}. 
    $$
\end{definition}
Note that $\Sigma_{1,\p}^\sigma$ is just the set of all \'etale algebras with splitting symbol $\sigma$. 
\begin{lemma}
    \label{lem-mass-formula-in-terms-of-symbol}
    Let $\sigma$ be a splitting symbol of degree $n$, and let $\p$ be any finite prime of $k$. Then 
    $$
    \m_\p(\Sigma_{1,\p}^\sigma) = \frac{1}{\Disc_\p(\sigma)\#\Aut(\sigma)}.
    $$
\end{lemma}
\begin{proof}
    By definition of $\Sigma_{\alpha,\p}^\sigma$ and $\tilde{m}_\p$, we have
    $$
    \m_\p(\Sigma_{1,\p}^\sigma) = \sum_{\substack{[L : k_\p] = n \text{ \'etale} \\ (L,\p) = \sigma}} \frac{1}{\Disc_\p(L)}\cdot \frac{1}{\#\Aut(L/k_\p)}.   
    $$ 
    Therefore, the result is essentially \cite[Proposition~2.1]{bhargava-mass-formula}, adapted to the case of general $p$-adic fields. The appendix of \cite{bhargava-mass-formula} explains how to make this adaptation. 
\end{proof}

\begin{lemma}
    \label{lem-masses-of-splitting-symbols}
    Let $\p$ be any finite prime of $k$. The pre-masses of the sets $\Sigma_{1,\p}^\sigma$ are given by the following table: 
    \begin{center}
        \begin{tabular}{|c|c|}
            \hline 
            $\sigma$ & $\m_\p(\Sigma_{1,\p}^\sigma)$
            \\
            \hline
            $ (1111) $ & $ \frac{1}{24} $
            \\
            \hline
            $ (112) $ & $ \frac{1}{4} $
            \\
            \hline
            $ (13) $ & $ \frac{1}{3} $
            \\
            \hline
            $ (22) $ & $ \frac{1}{8} $
            \\
            \hline
            $ (4) $ & $ \frac{1}{4} $
            \\
            \hline
            $ (1^211) $ & $ \frac{1}{2 \, q} $
            \\
            \hline
            $ (1^22) $ & $ \frac{1}{2 \, q} $
            \\
            \hline
            $ (1^21^2) $ & $ \frac{1}{2 \, q^{2}} $
            \\
            \hline
            $ (2^2) $ & $ \frac{1}{2 \, q^{2}} $
            \\
            \hline
            $ (1^31) $ & $ \frac{1}{q^{2}} $
            \\
            \hline
            $ (1^4) $ & $ \frac{1}{q^{3}} $
            \\
            \hline
        \end{tabular}
    \end{center}
\end{lemma}
\begin{proof}
    This follows immediately from Lemma~\ref{lem-mass-formula-in-terms-of-symbol}. 
\end{proof}
Write $\Sigma_{\p}^\triv$ for the set of isomorphism classes of quartic \'etale $k_\p$-algebras with trivial splitting symbols. It follows from Lemma~\ref{lem-triv-symbols-have-full-norm-group} that $\Sigma_\p^\triv \subseteq \Sigma_{\mathcal{A},\p}$ for all finitely generated subgroups $\mathcal{A}\subseteq k^*$. 

\begin{lemma}
    \label{lem-mass-triv}
    For all finite primes $\p$ of $k$, we have 
    $$
    \m_\p(\Sigma_{\p}^\triv) = \frac{5q^2 + 8q + 8}{8q^2}. 
    $$
\end{lemma}
\begin{proof}
    This follows immediately from Lemma~\ref{lem-masses-of-splitting-symbols}.
\end{proof}
We fix some notation for use in the following few lemmas. As always, let $\alpha \in k^*$. Let $\p$ be a finite prime of $k$, and let $q = N(\p)$. Write $r = \ord_\p(\alpha)$, and fix an arbitrary element $\pi \in \p\setminus \p^2$. Then $\alpha = u\pi^r$ for a unique $u \in k^*$, and moreover $u \in \co_{k_\p}^*$. The notations $r,\pi,$ and $u$ will be used without introduction for the rest of the current subsection.
\begin{lemma}
    \label{lem-mass-(22)}
    For all finite primes $\p$, we have 
    $$
    \m_\p(\Sigma_{\alpha,\p}^{(22)}) = \begin{cases}
        \frac{1}{8} \quad \text{if $2 \mid r$},
        \\
        0 \quad \text{otherwise}. 
    \end{cases}
    $$
\end{lemma}
\begin{proof}
    Lemma~\ref{lem-norm-group-(22)} tells us that 
    $$
    \Sigma_{\alpha,p}^{\sigma} = \begin{cases}
        \Sigma_{1,p}^{\sigma} \quad \text{if $2\mid r$},
        \\
        \varnothing \quad \text{otherwise},
    \end{cases}
    $$
    so the result follows from Lemma~\ref{lem-masses-of-splitting-symbols}.
\end{proof}
\begin{lemma}
    \label{lem-mass-(4)}
    For all finite primes $\p$, we have 
    $$
    \tilde{m}_\p(\Sigma_{\alpha,\p}^{(4)}) = \begin{cases}
        \frac{1}{4} \quad \text{if $4 \mid r$},
        \\
        0 \quad \text{otherwise}. 
    \end{cases}
    $$
\end{lemma}
\begin{proof}
    The proof is essentially the same as that of Lemma~\ref{lem-mass-(22)}, using Lemmas~\ref{lem-norm-group-(4)} and~\ref{lem-masses-of-splitting-symbols}.
\end{proof}
\begin{lemma}
    \label{lem-mass-(1^21^2)}
    If $q$ is odd, then we have 
    $$
    \m_\p(\Sigma_{\alpha,\p}^{(1^21^2)}) = \begin{cases}
        \frac{1}{2q^2} \quad \text{if $2 \mid r$ and $u \in \co_{k_\p}^{* 2}$},
        \\
        \frac{1}{4q^2} \quad \text{if $2 \mid r$ and $u \not \in \co_{k_\p}^{* 2}$},
        \\
        \frac{3}{8q^2} \quad \text{if $2\nmid r$}. 
    \end{cases}
    $$
\end{lemma}
\begin{proof}
    Lemma~\ref{lem-norm-group-(1^21^2)} tells us that $\Sigma_{1,\p}^{(1^21^2)} = \{L_1, L_2, L_3\}$, where the fields $L_j$ and their norm groups are given by the following table:
    \begin{center}
    \begin{tabular}{|c|c|c|}
        \hline
        $j$ & $L_j$ & $N_{L_j/k_\p}(L_j^*)$
        \\
        \hline \hline
        $1$ & $k_\p(\sqrt{\pi})\times k_\p(\sqrt{\pi})$ & $\{v^2(-\pi)^m : v \in \co_{k_\p}^*, m \in \Z\}$
        \\
        \hline 
        $2$ & $k_\p(\sqrt{\pi}) \times k_\p(\sqrt{\zeta_{q-1}\pi})$ & $k_\p^*$
        \\
        \hline
        $3$ & $k_\p(\sqrt{\zeta_{q-1}\pi}) \times k_\p(\sqrt{\zeta_{q-1}\pi})$ & $\{v^2(-\zeta_{q-1}\pi)^m : v \in \co_{k_\p}^*, m \in \Z\}$
        \\
        \hline
    \end{tabular}
\end{center}
    We have 
    $$
    \#\Aut(L_j/k_\p) = \begin{cases}
        8 \quad \text{if $j=1,3$},
        \\
        4 \quad \text{otherwise},
    \end{cases}
    $$
    and
    $$
    \Disc_\p(L_j) = q^2 \quad \text{for all $j$},
    $$
    by Lemma~\ref{lem-disc-from-symbol}. It follows that 
    $$
    \m_\p(\{L_j\}) = \begin{cases}
        \frac{1}{8q^2} \quad \text{if $j=1,3$},
        \\
        \frac{1}{4q^2} \quad \text{otherwise}.
    \end{cases}
    $$ 
    Some casework establishes that 
    $$
    \Sigma_{\alpha,\p}^{(1^21^2)} = \begin{cases}
        \{L_1,L_2,L_3\} \quad \text{if $2 \mid r$ and $u \in \co_{k_\p}^{* 2}$},
        \\
        \{L_2\} \quad \text{if $2 \mid r$ and $u \not \in \co_{k_\p}^{* 2}$},
        \\
        \{L_2,L_j\} \text{ for some $j\in\{1,3\}$} \quad \text{if $2\nmid r$}. 
    \end{cases}
    $$
    The result then follows by adding together the relevant pre-masses $\m_\p(\{L_j\})$. 
\end{proof}
The following three lemmas are proved in essentially the same way as Lemma~\ref{lem-mass-(1^21^2)}, using Lemmas~\ref{lem-norm-group-(2^2)}, \ref{lem-norm-group-(1^4)-1mod4}, and \ref{lem-norm-group-(1^4)-3mod4} respectively for the norm groups. 
\begin{lemma}
    \label{lem-mass-(2^2)}
    If $q$ is odd, then we have 
    $$
    \m_\p(\Sigma_{\alpha,\p}^{(2^2)}) = \begin{cases}
        \frac{1}{2q^2} \quad \text{if $4 \mid r$ and $u \in \co_{k_\p}^{* 2}$},
        \\
        0 \quad \text{if $4 \mid r$ and $u \not \in \co_{k_\p}^{* 2}$},
        \\
        \frac{1}{4q^2} \quad \text{if $r \equiv 2\pmod{4}$},
        \\
        0 \quad \text{if $2\nmid r$}.
    \end{cases}
    $$
\end{lemma}
\begin{lemma}
    \label{lem-mass-(1^4)-1mod4}
    If $q \equiv 1\pmod{4}$, then we have 
    $$
    \m_\p(\Sigma_{\alpha,\p}^{(1^4)}) = \begin{cases}
        \frac{1}{q^3} \quad \text{if $4 \mid r$ and $u \in \co_{k_\p}^{* 4}$},
        \\
        0 \quad \text{if $4 \mid r$ and $u \not \in \co_{k_\p}^{* 4}$},
        \\
        \frac{1}{2q^3} \quad \text{if $r \equiv 2\pmod{4}$ and $u \in \co_{k_\p}^{* 2}$},
        \\
        0 \quad \text{if $r \equiv 2 \pmod{4}$ and $u \not \in \co_{k_\p}^{* 2}$},
        \\
        \frac{1}{4q^3} \quad \text{if $2\nmid r$}.
    \end{cases}
    $$
\end{lemma}
\begin{lemma}
    \label{lem-mass-(1^4)-3mod4}
    If $q \equiv 3\pmod{4}$, then we have 
    $$
    \m_\p(\Sigma_{\alpha,\p}^{(1^4)}) = \begin{cases}
        \frac{1}{q^3} \quad \text{if $2\mid r$ and $u \in \co_{k_\p}^{* 2}$},
        \\
        0 \quad \text{if $2 \mid r$ and $u \not \in \co_{k_\p}^{* 2}$},
        \\
        \frac{1}{2q^3} \quad \text{if $2\nmid r$}.
    \end{cases}
    $$
\end{lemma}
We now have complete descriptions of the pre-masses $\m_\p(\Sigma_{\alpha,\p})$, and hence the masses $m_\p(\Sigma_{\alpha,\p})$, for finite primes $\p$ not lying over $2$. For infinite primes, the masses are given by the following two lemmas.
\begin{lemma}
    \label{lem-mass-real-arch}
    Let $f:k \to \C$ be a real embedding, and let $\p_f$ be the prime corresponding to $f$. Then 
    $$
    m_{\p_f}(\Sigma_{\alpha, \p_f}) = \begin{cases}
        \frac{5}{12} \quad \text{if $f(\alpha) > 0$},
        \\
        \frac{7}{24} \quad \text{if $f(\alpha) < 0$}. 
    \end{cases}
    $$
\end{lemma}
\begin{proof}
    Since $f$ is a real embedding, we have $k_{\p_f} \cong \R$. The algebraic extensions of $\R$ are precisely $\R$ and $\C$. The extension $\C/\R$ is totally ramified, so the possible quartic splitting symbols are $(1111),(111^2), (1^21^2)$. Each of these corresponds to a unique \'etale algebra, and the consequent norm groups and masses are given by the following table. 
    \begin{center}
        \begin{tabular}{|c|c|c|c|}
            \hline
            $\sigma$ & $L$ & $N_{L/k_\p}(L^*)$ & $m_\infty(\{L\})$
            \\
            \hline
            \hline
            $(1111)$ & $\R \times \R \times \R \times \R$ & $\R^*$ & $\frac{1}{24}$
            \\
            \hline
            $(111^2)$ & $\R\times\R\times\C$ & $\R^*$ & $\frac{1}{4}$
            \\
            \hline 
            $(1^21^2)$ & $\C\times\C$ & $\R_{>0}$ & $\frac{1}{8}$
            \\
            \hline
        \end{tabular}
    \end{center} 
    The masses in the table are computed using Definition~\ref{def-mass}. 
\end{proof}
\begin{lemma}
    \label{lem-mass-imaginary-arch}
    Let $g:k\to \C$ be a complex embedding, and let $\p_g$ be the prime corresponding to $g$. Then 
    $$
    m_{\p_g}(\Sigma_{\alpha,\p_g}) = \frac{1}{24}. 
    $$
\end{lemma}
\begin{proof}
    Since $k_{\p_g} \cong \C$ is algebraically closed, it has only one quartic \'etale algebra, up to isomorphism, and that \'etale algebra is isomorphic to $L = \C\times \C \times \C \times \C$. We have 
    $$
    \#\Aut(L/\C) = 24,
    $$
    so 
    $$
    m_{\p_g}(\{L\}) = \frac{1}{24}.
    $$
    We have $N_{L/\C}(L^*) = \C^*$, so $\Sigma_{\alpha,\p_g} = L$ for all $\alpha$.
\end{proof}
\begin{lemma}
    \label{lem-masses-of-odd-q}
    Let $\alpha \in k^*$ and let $\p$ be a finite prime of $k$ with $N(\p)$ odd. The values of $m_{\alpha,\p}$ are given by Tables~\ref{table-q-1mod4} and \ref{table-q-3mod4}. 
\end{lemma}
\begin{proof}
    Since 
    $$
    m_{\alpha,\p} = \sum_\sigma m_\p(\Sigma_{\alpha,\p}^\sigma),
    $$
    the result follows easily from Lemmas~\ref{lem-mass-triv}-\ref{lem-mass-(1^4)-3mod4}. 
\end{proof}

\subsection{Primes lying over $2$}
\label{subsec-primes-over-2}

Let $\p$ be a prime with even norm. The methods above do not apply, because in the wildly ramified case there are many more extensions. However, we have implemented an algorithm in MAGMA that computes $\widetilde{m}_\p(\Sigma_{\alpha,\p})$ for any number field $k$ and any finite prime $\p$ of $k$. This has two applications:
\begin{enumerate}
    \item The algorithm allows us to verify the results from Section~\ref{subsec-masses} for specific primes $\p$ with odd norm. Simultaneously, agreement with our results provides evidence that the algorithm works correctly. 
    \item The algorithm allows us to compute explicit values for $m_{\alpha,\p}(\Sigma_{\alpha,\p})$ for primes $\p$ lying over $2$. In particular, we used the algorithm to compute the values in Table~\ref{table-q-2}.
\end{enumerate}

Our code is in the GitHub repository at \url{https://github.com/Sebastian-Monnet/S4-quartics-prescribed-norms}. The most important function in our code is \newline{\tt{ComputePreMassOfNormSet}}. This takes as input a number {\tt{alpha}}, a $p$-adic field {\tt{BaseField}}, and a list {\tt{Symbols}} of splitting symbols. The function then explicitly computes the set of quartic \'etale algebras of {\tt{BaseField}} such that {\tt{alpha}} is a norm and returns the pre-mass of that set. In our notation, the function computes 
$$
\tilde{m}_\p\Big(\bigcup_{\sigma \in S}\Sigma_{\alpha,\p}^\sigma\Big),
$$ 
given a set $S$ of splitting symbols. 

The MAGMA file in the repository executes two functions: {\tt{PrintPreMassesFor2()}} and {\tt{TestUpTo(20,3)}}. The first of these functions prints the value of $m_2(\Sigma_{\alpha,2})$ for $\alpha = 2^ru$, as $r$ ranges over $\{0,1,2,3\}$ and $u$ ranges over $\{1,3,5,\ldots, 15\}$. We will see in the proof of Lemma~\ref{lem-masses-2} that this is precisely the set of values we need to compute in order to know $m_{\alpha,2}$ for all $\alpha\in\Q^*$. The second function checks, for each prime $p \in \{5,7,11,13,17,19\}$, that the values in Tables~\ref{table-q-1mod4} and~\ref{table-q-3mod4} are correct for all extensions $k_\p/\Q_p$ of degree  at most $3$. The function returns {\tt{true}} if the results agree with those we have proved, and {\tt{false}} otherwise.

\begin{lemma}
    \label{lem-fourth-powers-in-Q2}
    Let $\p$ lie over $2$ and let $\alpha \in \co_{k_\p}^*$. Suppose that $\alpha \equiv 1 \pmod{8\pi}$, for a uniformiser $\pi$ of $k_\p$. Then $\alpha \in \co_{k_\p}^{*4}$.
\end{lemma}
\begin{proof}
    Let $f(X) = (2X+1)^4 - \alpha$. Then we have 
    $$
    f(X) = 16X^4 + 32X^3 +24X^2 + 8X + 1 - \alpha,
    $$
    so $g(X) := \frac{1}{8}f(X) \in \co_{k_\p}[X]$, and we have 
    $$
    \lvert g(0) \rvert _\p = \Big\lvert \frac{1 - \alpha}{8} \Big\rvert _\p < 1,
    $$
    whereas 
    $$
    \lvert g'(0) \rvert _\p= \lvert 1 \rvert _\p = 1. 
    $$
    By Hensel's Lemma, the polynomial $g(X)$ has a root in $\co_{k_\p}$, hence $f(X)$ does too.
\end{proof}
\begin{lemma}
    \label{lem-masses-2}
    When $k=\Q$, the masses $m_{\alpha,2}$ are given by Table~\ref{table-q-2}.
\end{lemma}
\begin{proof}
    Lemma~\ref{lem-fourth-powers-in-Q2} implies that a system of representatives for $\Q_2^*/\Q_2^{*4}$ is given by $2^ru$, as $r$ ranges over $\{0,1,2,3\}$ and $u$ ranges over $\{1,3,5,\ldots, 15\}$. Since $\Q_2^{*4}\subseteq N_{L/\Q_2}(L^*)$ for any quartic \'etale algebra $L/\Q_2$, it suffices to compute 
    $$
    m_{2}(\Sigma_{\alpha,2})
    $$
    for each such $\alpha = 2^ru$. The corresponding pre-masses are outputted by the code in our repository, and we obtain the masses by halving the pre-masses.  
\end{proof}

\begin{proof}[Proof of Theorem~\ref{thm-values-of-masses}]
    The result is immediate from Lemmas~\ref{lem-masses-of-odd-q}, \ref{lem-mass-real-arch}, \ref{lem-mass-imaginary-arch}, and \ref{lem-masses-2}. 
\end{proof}
\section{Proportion of Extensions with Prescribed Norms}
\label{sec-proportion}
The goal of this section is to prove Theorem~\ref{thm-prop-subgroup-between-0-1}. Recall that, given a finite prime $\p$ of $k$, we often write $q$ for the norm $N(\p)$ of $\p$.  
\begin{lemma}
    \label{lem-mass-upper-bound}
    Let $\alpha \in k^*$, and let $\p$ be a finite prime of $k$ such that the following two conditions hold:
    \begin{enumerate}
        \item $q$ is odd,
        \item $4 \mid \ord_\p(\alpha)$. 
    \end{enumerate}
    Then $m_{\alpha,\p} < 1 + \frac{1}{q^2}$. 
\end{lemma}
\begin{proof}
    This is immediate from Tables~\ref{table-q-1mod4} and~\ref{table-q-3mod4}.
\end{proof}
Let $\alpha \in k^*$. In the following lemma, we write $S$ for the set of primes $\p$ of $k$ such that at least one of the following holds:
\begin{enumerate}
    \item $\p$ is finite and $q$ is even,
    \item $\p$ is finite and $4 \nmid \ord_\p(\alpha)$,
    \item $\p$ is infinite.
\end{enumerate}
Then $S$ is a finite set, so we can explicitly compute the mass $m_{\alpha,\p}$ for every $\p \in S$. 
\begin{lemma}
    \label{lem-lim-finite}
    Let $\alpha \in k^*$, and let $S$ be the set of primes defined above. We have an explicit finite bound
    $$
    \lim_{X\to\infty}\frac{N(X;\alpha)}{X} \leq \frac{1}{2}\operatornamewithlimits{Res}_{s=1}\zeta_k(s)\cdot \zeta_{k,S}(2)\cdot\prod_{\p \in S}m_{\alpha,\p},
    $$
    where $\zeta_{k,S}(s) = \prod_{p\not \in S}(1 - (N\p)^{-s})^{-1}$ is the partial zeta function.
\end{lemma}
\begin{proof}
    It is easy to see that $S$ is a finite set, so we have
    $$
    \prod_{\p \in S} m_{\alpha,\p} < \infty.
    $$
    By Theorem~\ref{thm-counting-function-subgroup}, we have 
    $$
    \lim_{X\to\infty}\frac{N(X;\alpha)}{X} = \frac{1}{2}\operatornamewithlimits{Res}_{s=1}\zeta_k(s) \prod_{\p \in S}m_{\alpha,\p} \cdot \prod_{\p \not \in S}m_{\alpha,\p},
    $$
    so it suffices to show that 
    $$
    \prod_{\p\not\in S}m_{\alpha,\p} \leq \zeta_{k,S}(2).
    $$
    Note that $S$ is precisely the set of primes to which Lemma~\ref{lem-mass-upper-bound} does not apply, so we have 
    \begin{align*}
        \prod_{\p\not\in S}m_{\alpha,\p} &\leq \prod_{\p\not\in S} \Big(1 + \frac{1}{N(\p)^2}\Big)
        \\
        &\leq \prod_{\p \not \in S} \Big(1 + \frac{1}{N(\p)^2} + \frac{1}{N(\p)^4} + \ldots \Big)
        \\
        &= \zeta_{k,S}(2). 
    \end{align*}
\end{proof}

\begin{lemma}
    \label{lem-subgroup-all-masses-positive}
    Let $\mathcal{A}\subseteq k^*$ be a finitely generated subgroup and let $\p$ be any prime of $k$. Then $m_{\mathcal{A},\p} > 0$.
\end{lemma}
\begin{proof}
    If $\p$ is archimedean, the result follows from Lemmas~\ref{lem-mass-real-arch} and \ref{lem-mass-imaginary-arch}. 
    If $\p$ is finite, then 
    $$
    m_\p(\Sigma_{\mathcal{A},\p}) \geq m_\p(\Sigma_\p^\triv) = \frac{q-1}{q}\cdot\frac{5q^2 + 8q + 8}{8q^2},
    $$
    by Lemma~\ref{lem-mass-triv}. 
\end{proof}
\begin{lemma}
    \label{lem-subgroup-counting-function-lim-positive}
    For any finitely generated subgroup $\mathcal{A}\subseteq k^*$, we have 
    $$
    \lim_{X\to\infty} \frac{N(X;\mathcal{A})}{X} > 0. 
    $$
\end{lemma}
\begin{proof}
    Let $\alpha_1,\ldots, \alpha_n$ be a finite generating set for $\mathcal{A}$, so that $\Sigma_{\mathcal{A},\p} = \cap_i \Sigma_{\alpha_i,\p}$ by Lemma \ref{lem-sigma-of-subgroup-eq-intersection}. Let $S$ be the set of primes of $k$ such that at least one of the following conditions holds:
    \begin{enumerate}
        \item $\p$ is finite and $q$ is even,
        \item $\p$ is finite and $q = 3$,
        \item $4\nmid \ord_\p(\alpha_i)$ for some $i$,
        \item $\p$ is infinite. 
    \end{enumerate}
    It is easy to see that $S$ is a finite set. Moreover, if $\p \not \in S$, then 
    $$
    \Sigma_{\mathcal{A},\p} \supseteq \Sigma_{\p}^\triv \cup \Sigma_{1,\p}^{(22)} \cup \Sigma_{1,\p}^{(4)},
    $$
    by Lemmas~\ref{lem-triv-symbols-have-full-norm-group}, \ref{lem-norm-group-(22)}, and \ref{lem-norm-group-(4)}.
    Lemmas~\ref{lem-mass-triv},~\ref{lem-mass-(22)}, and~\ref{lem-mass-(4)} tell us that, for all finite $\p$, we have
    $$
    \m_\p(\Sigma_{\p}^\triv \cup \Sigma_{1,\p}^{(22)} \cup \Sigma_{1,\p}^{(4)}) = \frac{q^2 + q + 1}{q^2}.
    $$
    It follows that 
    $$
    m_\p(\Sigma_{\p}^\triv \cup \Sigma_{1,\p}^{(22)} \cup \Sigma_{1,\p}^{(4)}) =  1 - \frac{1}{q^3}. 
    $$ 
    Taking the product over all finite primes $\p$ of $k$, we have 
    $$
    \prod_{\p\nmid \infty} \Big(1 - \frac{1}{N(\p)^3}\Big) = \zeta_k(3)^{-1},
    $$
    which is greater than $0$. Since $S$ is finite, it follows that 
    $$
    \prod_{\p \not \in S} m_{\mathcal{A},\p} \geq \prod_{\p \not \in S} \Big(1 - \frac{1}{N(\p)^3}\Big) > 0.
    $$
    Since $S$ is finite, Lemma~\ref{lem-subgroup-all-masses-positive} implies that 
    $$
    \prod_{\p \in S} m_{\mathcal{A},\p} > 0. 
    $$
    The result now follows from Theorem~\ref{thm-counting-function-subgroup}.
\end{proof}
\begin{lemma}
    \label{lem-subgroup-counting-function-lim-finite}
    For any finitely generated subgroup $\mathcal{A}\subseteq k^*$, we have 
    $$
    \lim_{X\to\infty} \frac{N(X;\mathcal{A})}{X} < \infty . 
    $$
\end{lemma}
\begin{proof}
    This follows immediately from Lemma~\ref{lem-lim-finite}, since $N(X;\mathcal{A}) \leq N(X;\alpha)$ for all $\alpha \in \mathcal{A}$. 
\end{proof}
The following lemma is elementary.
\begin{lemma}[Elementary analysis lemma]
    \label{lem-analysis-lem}
    Let $(a_n)$ and $(b_n)$ be sequences of real numbers, and suppose that $A = \prod_{n=1}^\infty a_n$ and $B = \prod_{n=1}^\infty b_n$ are conditionally convergent products. Then the product
    $$
    \prod_{n=1}^\infty \frac{a_n}{b_n}
    $$
    is conditionally convergent to $\frac{A}{B}$. 
\end{lemma}
\begin{lemma}
    \label{lem-product-for-prop}
    We have 
    $$
    \lim_{X\to\infty} \frac{N(X;\mathcal{A})}{N(X)} = \prod_{\p} \frac{m_{\mathcal{A},\p}}{m_{1,\p}}.
    $$
\end{lemma}
\begin{proof}
    This follows immediately from Theorem~\ref{thm-counting-function-subgroup} and Lemma~\ref{lem-analysis-lem}, whose hypotheses are satisfied by Lemmas~\ref{lem-subgroup-counting-function-lim-finite} and~\ref{lem-subgroup-counting-function-lim-positive}.
\end{proof}
\begin{proof}[Proof of Theorem~\ref{thm-prop-subgroup-between-0-1}]
    For each finitely generated subgroup $\mathcal{A}\subseteq k^*$, Lemmas~\ref{lem-subgroup-counting-function-lim-positive} and~\ref{lem-subgroup-counting-function-lim-finite} tell us that 
    $$
    0 < \lim_{X\to\infty} \frac{N(X;\mathcal{A})}{X} < \infty. 
    $$
    It follows by the algebra of limits that 
    $$
    \Bigg(\lim_{X\to\infty}\frac{N(X;\mathcal{A})}{X}\Bigg) / \Bigg(\lim_{X\to\infty}\frac{N(X;\{1\})}{X}\Bigg) = \lim_{X\to\infty}\frac{N(X;\mathcal{A})}{N(X;\{1\})}.
    $$
    Lemma~\ref{lem-subgroup-counting-function-lim-positive} tells us that the left-hand side is a ratio of two positive numbers, so we have 
    $$
    \lim_{X\to\infty}\frac{N(X;\mathcal{A})}{N(X;\{1\})} > 0.
    $$
    Clearly 
    $$
    \lim_{X\to\infty}\frac{N(X;\mathcal{A})}{N(X;\{1\})} \leq 1,
    $$
    with equality if and only if 
    $$
    \lim_{X\to\infty}\frac{N(X;\mathcal{A})}{X}=\lim_{X\to\infty}\frac{N(X;\{1\})}{X}.
    $$
    By Lemma~\ref{lem-product-for-prop}, this equality holds if and only if $m_{\mathcal{A},\p} = m_{1,\p}$ for all $\p$, and this is the case if and only if $\Sigma_{\mathcal{A},\p} = \Sigma_{1,\p}$ for all $\p$, hence if and only if $\mathcal{A}\subseteq N_{L/k_\p}(L^*)$ for every $\p$ and every quartic \'etale algebra $L/k_\p$. 
    
    Suppose that this is the case. Taking $L$ to be the unramified quartic extension of $k_\p$ shows that $4\mid v_\p(\alpha)$ for all $\alpha \in \mathcal{A}$ and for each prime $\p$ of $k$. Let $T$ be the set of primes $\p$ of $k$ with $N(\p)$ odd. Let $\alpha \in \mathcal{A}$ and let $\p \in T$. Write $\pi$ for a uniformiser of $\co_{k_\p}$, so that $\alpha = \pi^ru$, where $4\mid r$ and $u \in \co_{k_\p}^*$. Taking $L$ to be a totally ramified quartic field extension of $k_\p$, Lemmas~\ref{lem-norm-group-(1^4)-1mod4} and \ref{lem-norm-group-(1^4)-3mod4} tell us that $u \in \co_{k_\p}^{*4}$. Therefore, $\alpha$ is in $k_\p^{*4}$ for all $\p \in T$, so \cite[Theorem~9.1.11]{neukirch-cohomology-of-nfs} tells us that $\alpha \in k^{*4}$. We have shown that 
    $$
    \lim_{X\to\infty} \frac{N(X;\mathcal{A})}{N(X;\{1\})} = 1 \implies \mathcal{A}\subseteq k^{*4}.
    $$
    The converse implication is trivial, so we are done. 
\end{proof}
\appendix
\section{Tables of Masses}
To use Tables~\ref{table-q-1mod4} and \ref{table-q-3mod4}, let $r = \ord_\p(\alpha)$ and choose $\pi \in \p\setminus \p^2$. Then set $u = \pi^{-r}\alpha$. Let 
$$
\bar{r} = r\pmod{4}, \quad \bar u = u\pmod{\p} \in \F_\p.
$$
Then $m_{\alpha,\p}$ is the entry in the table corresponding to $\bar{r}$ and the status of $\bar{u}$ as a fourth power or a square. 

To use Table~\ref{table-q-2}, write $\alpha = 2^ru$ for an odd integer $u$. Let 
$$
    \bar{r} = r\pmod{4},\quad \bar{u} = u\pmod{16}.
$$
Then $m_{\alpha,2}$ is the corresponding entry of the table. 
\begin{table}[H]
    \caption{$q \equiv 1\pmod{4}$}
    \label{table-q-1mod4}
    \begin{center}
        \begin{tabular}{|c||c|c|c|c|}
        \hline \diagbox{$\bar{u}$}{$\bar{r}$} & 0 & 1 & 2 & 3 \\ \hline\hline
        $\F_\p^{* 4}$& $\frac{{\left(q^{3} + q^{2} + 2 \, q + 1\right)} {\left(q - 1\right)}}{q^{4}}$ & $\frac{{\left(5 \, q^{3} + 8 \, q^{2} + 11 \, q + 2\right)} {\left(q - 1\right)}}{8 \, q^{4}}$ & $\frac{{\left(q^{2} + q + 2\right)} {\left(3 \, q + 1\right)} {\left(q - 1\right)}}{4 \, q^{4}}$ & $\frac{{\left(5 \, q^{3} + 8 \, q^{2} + 11 \, q + 2\right)} {\left(q - 1\right)}}{8 \, q^{4}}$ 
        \\
        \hline
        $\F_\p^{* 2} \setminus \F_\p^{* 4}$& $\frac{{\left(q^{2} + q + 2\right)} {\left(q - 1\right)}}{q^{3}}$ & $\frac{{\left(5 \, q^{3} + 8 \, q^{2} + 11 \, q + 2\right)} {\left(q - 1\right)}}{8 \, q^{4}}$ & $\frac{{\left(q^{2} + q + 2\right)} {\left(3 \, q + 1\right)} {\left(q - 1\right)}}{4 \, q^{4}}$ & $\frac{{\left(5 \, q^{3} + 8 \, q^{2} + 11 \, q + 2\right)} {\left(q - 1\right)}}{8 \, q^{4}}$ 
        \\
        \hline
        $\F_\p^* \setminus \F_\p^{* 2}$& $\frac{{\left(4 \, q^{2} + 4 \, q + 5\right)} {\left(q - 1\right)}}{4 \, q^{3}}$ & $\frac{{\left(5 \, q^{3} + 8 \, q^{2} + 11 \, q + 2\right)} {\left(q - 1\right)}}{8 \, q^{4}}$ & $\frac{{\left(3 \, q^{2} + 4 \, q + 6\right)} {\left(q - 1\right)}}{4 \, q^{3}}$ & $\frac{{\left(5 \, q^{3} + 8 \, q^{2} + 11 \, q + 2\right)} {\left(q - 1\right)}}{8 \, q^{4}}$ 
        \\
        \hline
        \end{tabular}\end{center}
\end{table}
\begin{table}[H]
    \caption{$q\equiv 3\pmod{4}$}
    \label{table-q-3mod4}
    \begin{center}
        \begin{tabular}{|c||c|c|c|c|}
        \hline \diagbox{$\bar{u}$}{$\bar{r}$} & 0 & 1 & 2 & 3 \\ \hline\hline
        $\F_\p^{* 2}$& $\frac{{\left(q^{3} + q^{2} + 2 \, q + 1\right)} {\left(q - 1\right)}}{q^{4}}$ & $\frac{{\left(5 \, q^{3} + 8 \, q^{2} + 11 \, q + 4\right)} {\left(q - 1\right)}}{8 \, q^{4}}$ & $\frac{{\left(3 \, q^{3} + 4 \, q^{2} + 7 \, q + 4\right)} {\left(q - 1\right)}}{4 \, q^{4}}$ & $\frac{{\left(5 \, q^{3} + 8 \, q^{2} + 11 \, q + 4\right)} {\left(q - 1\right)}}{8 \, q^{4}}$ 
        \\
        \hline
        $\F_\p^* \setminus \F_\p^{* 2}$& $\frac{{\left(4 \, q^{2} + 4 \, q + 5\right)} {\left(q - 1\right)}}{4 \, q^{3}}$ & $\frac{{\left(5 \, q^{3} + 8 \, q^{2} + 11 \, q + 4\right)} {\left(q - 1\right)}}{8 \, q^{4}}$ & $\frac{{\left(3 \, q^{2} + 4 \, q + 6\right)} {\left(q - 1\right)}}{4 \, q^{3}}$ & $\frac{{\left(5 \, q^{3} + 8 \, q^{2} + 11 \, q + 4\right)} {\left(q - 1\right)}}{8 \, q^{4}}$ 
        \\
        \hline
        \end{tabular}\end{center} 
\end{table}
\begin{table}[H]
    \caption{Table defining $m_{\alpha,2}$.}
    \label{table-q-2}
\begin{center}
    \begin{tabular}{|c||c | c | c | c|}
    \hline
    \diagbox{$\bar{u}$}{$\bar{r}$}& 0 & 1 & 2 & 3
    \\ \hline \hline
    1& $17/16$& $6523/8192$& $3791/4096$& $6523/8192$
    \\
    \hline
    3& $65/64$& $6523/8192$& $115/128$& $6523/8192$
    \\
    \hline
    5& $535/512$& $6523/8192$& $469/512$& $6523/8192$
    \\
    \hline
    7& $4159/4096$& $6523/8192$& $3679/4096$& $6523/8192$
    \\
    \hline
    9& $2175/2048$& $6523/8192$& $3791/4096$& $6523/8192$
    \\
    \hline
    11& $65/64$& $6523/8192$& $115/128$& $6523/8192$
    \\
    \hline
    13& $535/512$& $6523/8192$& $469/512$& $6523/8192$
    \\
    \hline
    15& $4159/4096$& $6523/8192$& $3679/4096$& $6523/8192$
    \\
    \hline
    \end{tabular} 
    \end{center}
\end{table}

\printbibliography
\end{document}